\documentclass[a4paper,11pt,a4wide]{article}

\usepackage[sc]{mathpazo}
\usepackage[T1]{fontenc}
\usepackage[utf8]{inputenc}
\usepackage{amsfonts}
\usepackage{amssymb}
\usepackage{amsthm}
\usepackage{amsmath}

\usepackage{hyperref}

\usepackage{eurosym}
\usepackage{graphicx}
\usepackage{colortbl}
\usepackage{epstopdf}
\usepackage[usenames,dvipsnames]{xcolor}

\newcommand*{\N}{\mathbb{N}}
\newcommand*{\C}{\mathbb{C}}

\newcommand*{\prob}{\mathbb{P}}

\newcommand*{\eps}{\varepsilon}

\newtheorem{theorem}{Theorem}
\newtheorem{lemma}{Lemma}
\theoremstyle{definition}
\newtheorem{remark}{Remark}
\newtheorem{example}{Example}
\newtheorem{question}{Question}

\title{Absence of zeros implies strong spatial mixing}
\author{Guus Regts\thanks{Korteweg de Vries Institute for Mathematics, University of Amsterdam, the Netherlands. Email: \texttt{guusregts@gmail.com}. Supported by NWO Vidi grant VI.Vidi.193.068}}

\begin{document}

\maketitle
\begin{abstract}
In this paper we show that absence of complex zeros of the partition function of the hard-core model on any family of bounded degree graphs that is closed under taking induced subgraphs implies that the associated probability measure, the \emph{hard-core measure}, satisfies strong spatial mixing on that family. 
As a corollary we obtain that the hard-core measure on the family of bounded degree claw-free graphs satisfies strong spatial mixing for every value of the fugacity parameter. 
We furthermore derive strong spatial mixing for graph homomorphism measures from absence of zeros of the graph homomorphism partition function.
\end{abstract}
\section{Introduction}
The design of efficient algorithms to (approximately) compute evaluations of partition functions and graph polynomials, such as the matching polynomial, the independence polynomial, the number of proper colorings and more generally the partition function of the Potts model, is an active area of research.
There are two main approaches to obtain deterministic algorithms for this task. 
One is based on a notion of decay of correlations, related to the absence of phase transitions in statistical physics, called \emph{strong spatial mixing}. 
This method was pioneered by Weitz~\cite{Weitz06} and Bandyopadhyay and Gamarnik~\cite{BandyopadhyayGamarnik} and dates back about fifteen years. 
The other is the \emph{interpolation method} of Barvinok~\cite{BarvinokBook} in combination with an algorithm of Patel and the author~\cite{PatelRegts17}, which is based on absence of complex zeros of the partition function and relates to absence of phase transitions in the Lee-Yang~\cite{LeeYang1} sense.

Let us for concreteness give an example to illustrate some of these notions.
\begin{example}[The hard-core model]\label{ex:ind pol}
Let $G=(V,E)$ be a graph and $\lambda\in \C$. The \emph{independence polynomial} of $G$ evaluated at $\lambda$ is given by    \begin{equation}\label{eq:def ind pol}
    Z_G(\lambda)=\sum_{\substack{I\subseteq V\\I \text{ independent}}}\lambda^{|I|}.
\end{equation}
where a set $I\subseteq V$ is called independent if it does not span any edge of $G$.
In statistical physics $Z_G(\lambda)$ is known as the \emph{partition function of the hard-core model} and $\lambda$ is called the \emph{fugacity}.
For positive $\lambda$ there is a natural associated probability measure, $\mu_{G,\lambda}$, on the collection of all independent sets of $G$, which is called the \emph{hard-core measure} and is defined by
\[
\mu_{G,\lambda}(I)=\frac{\lambda^{|I|}}{Z_G(\lambda)},
\]
for an independent set $I$. Often we just write $\mu$ instead of $\mu_{G,\lambda}$.

Let for a positive integer $\Delta$, $\mathcal{G}_\Delta$ be the family of graphs of maximum degree at most $\Delta$.
If for any $G\in \mathcal{G}_\Delta$ and any two vertices $u,v\in V(G)$, 
\[
\left|\Pr_{\mu}[u,v\in I]-\Pr_{\mu}[u\in I]\Pr_{\mu}[v\in I]\right|< \delta(d_G(u,v)),
\]
where $\delta:\mathbb{N}\to [0,\infty)$ is a function that goes to $0$ as its input goes to infinity and $d_G(u,v)$ denotes the graph distance between the vertices $u$ and $v$ in $G$, then we say that $\mu_{G,\lambda}$ satisfies \emph{(point to point) decay of correlations with rate $\delta$} on $\mathcal{G}_\Delta$. 

Weitz~\cite{Weitz06} showed that for $\lambda\in (0,\lambda_c)$, where $\lambda_c=\frac{(\Delta-1)^{\Delta-1}}{(\Delta-2)^\Delta}$, $\mu_{G,\lambda}$ satisfies a stronger form of decay of correlation called \emph{strong spatial mixing}, which we will formally define below, and used this to device a deterministic polynomial time approximation algorithm for computing $Z_G(\lambda)$ for $G\in \mathcal{G}_\Delta$ and $\lambda\in (0,\lambda_c)$\footnote{Technically, strong spatial mixing does not automatically yield efficient algorithms. Weitz~\cite{Weitz06} in fact proved something slightly stronger than strong spatial mixing.}. 

Peters and the author~\cite{PetersRegts19} showed that there exists an open set $U\subset \C$ containing the interval $(0,\lambda_c)$ such that for all $\lambda\in U$ and $G\in \mathcal{G}_\Delta$, $Z_G(\lambda)\neq 0$. Combined with Barvinok's interpolation method~\cite{BarvinokBook,PatelRegts17} this also yields a deterministic polynomial time approximation algorithm for computing $Z_G(\lambda)$ for $G\in \mathcal{G}_\Delta$ and $\lambda\in (0,\lambda_c)$.
\end{example}

Despite the difference in the respective approaches, surprisingly both approaches have given comparable results in many situations. Not just for the independence polynomial as mentioned in the example above, but also for the matching polynomial~\cite{Bayatietalmatchings,PatelRegts17}, the edge cover polynomial~\cite{edge,weighted,bencs2020some} and the graph homomorphism partition function~\cite{lu2013improved,BarvinokSoberon,BarvinokBook}. This begs the question of how these two approaches are related.

Many of the above mentioned polynomials and partition functions originate in statistical physics, where they are typically studied on structured subgraphs of lattices such as $\mathbb{Z}^d$. 
Dobrushin and Shlossmann~\cite{DS84,DS87} came up with an extensive list of equivalent characterization of what they call completely analytical interactions, in particular showing that absence of zeros and (some forms of) decay of correlations are equivalent for models like the hard-core model and many others.
Their proof depends strongly on the fact that balls of radius $r$ in the graph $\mathbb{Z}^d$ (for a fixed $d$) grow only polynomially with $r$. 
This is of course not true in general bounded degree graphs. So the work of Dobrushin and Shlossmann only gives a suggestion of what could be true for other families of graphs.

Understanding the connection between strong spatial mixing and absence of zeros on families of graphs like $\mathcal{G}_\Delta$ has recently started to receive  attention~\cite{LSSFisherzeros,LSScorrelation,shao2019contraction,gamarnik2020correlation}. In particular in~\cite{LSScorrelation,shao2019contraction} it is shown that a standard method for proving strong spatial mixing can be used to prove absence of zeros for partition functions of several models.
Very recently, Gamarnik~\cite{gamarnik2020correlation} showed that absence of zeros of the partition function implies a weaker form of strong spatial mixing for the hardcore model and certain graph homomorphism models, but his result does not apply to all bounded degree graphs.

In the present paper we will show that absence of zeros of the partition function does indeed imply strong spatial mixing for the hardcore model and certain graph homomorphism models for all bounded degree graphs, confirming in a strong form a variant of a conjecture of Gamarnik~\cite{gamarnik2020correlation}.

Below we shall give formal definitions of the notion of strong spatial mixing that we use and state our main results.
For concreteness we will limit ourselves to two types of models: the hard-core model and graph homomorphisms. 
We shall later indicate how our approach can be used for other models as well.
\subsection{The hard-core model}
We continue our discussion of Example~\ref{ex:ind pol}.
To introduce the notion of strong spatial mixing we need to consider \emph{boundary conditions}. 
For a graph $G=(V,E)$ and $\Lambda\subset V$ we call $\sigma:\Lambda\to \{0,1\}$ a \emph{boundary condition} if $\sigma^{-1}(1)$ is an independent set in the graph induced by $\Lambda$.
We denote by $\Pr_\mu[v\in I\mid \sigma]$ the probability that the vertex $v$ is in the random independent set $I$ drawn according to the hard-core measure conditioned on $\sigma$, meaning that we condition on the event that $\sigma^{-1}(1)\subset I$ and $\sigma^{-1}(0)\cap I=\emptyset$.
For another boundary condition $\tau$ on $\Lambda$ and a vertex $v\notin \Lambda$, we denote by $d_G(v,\sigma\neq \tau)$ the graph distance from $v$ to the nearest vertex in $\Lambda$ at which $\sigma$ and $\tau$ differ.

Let $\mathcal{G}$ be an infinite family of graphs and let $\lambda>0$.
The hard-core measure at $\lambda$ satisfies \emph{strong spatial mixing on $\mathcal{G}$} with exponential rate $r>1$ if there exists a constant $C>0$ such that for any graph $G=(V,E)\in \mathcal{G}$, any vertex $v\in V$, any $\Lambda\subseteq V\setminus \{v\}$ and any two boundary conditions $\sigma$ and $\tau$ on $\Lambda$,
\begin{equation}\label{eq:def SSM ind}
\left|\Pr_\mu[v\in I\mid \sigma]-\Pr_\mu[v\in I\mid \tau]\right|\leq Cr^{-d_G(v,\sigma\neq \tau)}.
\end{equation}
Note that strong spatial mixing implies point to point correlation with an exponentially decaying rate.

For a set $S\subset\C$ and $\eps>0$ we denote $\mathcal{N}(S,\eps)=\{z\in \C\mid d(z,S)\leq \eps\}$, where $d$ denotes the Euclidean metric on $\C$.
We can now state our main result for the hard-core measure, which we prove in Section~\ref{sec:ind}.
\begin{theorem}\label{thm:main general}
Let $\Delta\geq 2$ be an integer and let $\mathcal{G}\subset\mathcal{G}_\Delta$ be a family of bounded degree graphs that is closed under taking induced subgraphs.
Let $\lambda^\star>0$ be such that there exists $\eps>0$ such that for each $G\in \mathcal{G}$ and any $\lambda\in \mathcal{N}([0,\lambda^\star),\eps)$, $Z_G(\lambda)\neq 0$. 
Then for any $\lambda\in (0,\lambda^\star]$ the hard-core measure at $\lambda$ satisfies strong spatial mixing on $\mathcal{G}$ with exponential rate $r=1+\exp(-O(\lambda/\eps))$.
\end{theorem}

A celebrated result by Chudnovsky and Seymour~\cite{ChudnovskySeymour07} states that for any claw-free graph $G$ the zeros of $Z_G$ are all real and hence negative. Since for bounded degree graphs the zeros of its independence polynomial do not approach $0$ by a result of Shearer~\cite{shearer} and Scott and Sokal~\cite{ScottSokal05} (cf. Lemma~\ref{lem:bounded ratios near zero} below), we can apply our main theorem to the family of bounded degree claw-free graphs (which is certainly closed under taking induced subgraphs) to obtain an improvement on a result implicit in~\cite{Bayatietalmatchings}.
We can however get a much better exponential rate as the following result states.
\begin{theorem}\label{thm:main claw}
Let $\Delta\geq 2$ be an integer and let $\mathcal{G}\subset\mathcal{G}_\Delta$ be the family of claw-free graphs of maximum degree at most $\Delta$. Then for any $\lambda>0$, the hard-core measure at $\lambda$ satisfies strong spatial mixing on $\mathcal{G}$ with exponential rate $r=1+O((\lambda\Delta)^{-1/2})$.
\end{theorem}
We prove this result in Section~\ref{sec:ind}.

\subsubsection*{Algorithms}
As remarked in~\cite{Bayatietalmatchings}, strong spatial mixing by itself is not sufficient to approximately compute the probabilities $\Pr_\mu[v\in I]$ in polynomial time (in case one wants the additive error to be at most order $1/|V(G)|$). 
Our approach for showing strong spatial mixing implicitly yields a polynomial time algorithms for this task. We comment on this at the end of Section~\ref{sec:ind}.


\subsection{Graph homomorphism measures}

Let $q\geq 2$ be an integer and let $A$ be a symmetric $q\times q$ matrix.
For a graph $G=(V,E)$ we define the \emph{graph homorphism partition function} $Z_G(A)$ by
\begin{equation}\label{eq:graph hom def}
    Z_G(A):=\sum_{\psi\to [q]}\prod_{uv\in E} A_{\psi(u),\psi(v)},
\end{equation}
where $[q]:=\{1,\ldots,q\}.$ Note that in case $A$ is the adjacency matrix of a graph $H$, then $Z_A(G)$ is equal to the number of graph homomorphisms from $G$ to $H$.
For nonnegative (and nonzero) matrices $A$ there is a natural associated probability measure, the \emph{graph homomorphism measure} $\mu_{G,A}$, on the set of all $q$-colorings of the vertices of $G$, $\Omega_{V,q}=\{\psi: V\to [q]\}$, defined by for $\psi\in \Omega_{V,q}$,
\[
\mu_{G,A}(\psi):=\frac{\prod_{uv\in E}A_{\psi(u),\psi(v)}}{Z_G(A)},
\]
where we implicitly assume that $Z_G(A)\neq 0$.
For $\Lambda\subset V$ we call any $\sigma:\Lambda\to[q]$ a \emph{boundary condition} on $\Lambda$.
Let $v\in V\setminus \Lambda.$
We denote for $i\in [q]$ by $\Pr_\mu[\psi(v)=i \mid \sigma]$ the probability that the vertex $v$ gets color $i$ in the random $q$-coloring $\psi$ drawn according to the measure $\mu_{G,A}$ conditioned on $\sigma$, meaning that we condition on the event that $\psi$ agrees with $\sigma$ on $\Lambda$, where we implicitly assume that this latter event has positive measure.
As in the case of the independence polynomial, for another boundary condition $\tau$ on $\Lambda$, we denote by $d_G(v,\sigma\neq \tau)$ the graph distance from $v$ to the nearest vertex in $\Lambda$ at which $\sigma$ and $\tau$ differ.

Let $\mathcal{G}$ be an infinite family of graphs.
We say that the measure $\mu=\mu_{G,A}$ satisfies \emph{strong spatial mixing on $\mathcal{G}$} with exponential rate $r>1$ if there exists a constant $C>0$ such that for any graph $G=(V,E)\in \mathcal{G}$, any vertex $v\in V$, any $i\in [q]$, any $\Lambda\subseteq V\setminus \{v\}$ and any two boundary conditions $\sigma$ and $\tau$ on $\Lambda$,
\begin{equation}\label{eq:def SSM hom}
\left|\Pr_\mu[\psi(v)=i\mid \sigma]-\Pr_\mu[\psi(v)=i\mid \tau]\right|\leq Cr^{-d_G(v,\sigma\neq \tau)}.
\end{equation}

Recall that by $\mathcal{G}_\Delta$ we denote the family of graphs of maximum degree at most $\Delta$.
Using Barvinok's~\cite[Theorem 7.1.4]{BarvinokBook}, or rather its proof, which provides a zero-free region for the graph homomorphism partition function, we prove in Section~\ref{sec:hom} the following result. 

\begin{theorem}\label{thm:main hom}
Let $\Delta\geq 3$ and $q\geq 2$ be integers and let for some $\alpha=\alpha_\Delta<2\pi/3\Delta$,
\[
\delta_\Delta:=\sin(\alpha/2)\cos(\alpha\Delta/2).
\]
Fix $\eta\in (0,1)$. Then
for any real $q\times q$ symmetric matrix $A$ satisfying $|A_{i,j}-1|<(1-\eta)\delta_\Delta$ for all $i,j=1,\ldots,q$ the measure $\mu_{G,A}$ satisfies strong spatial mixing on $\mathcal{G}_\Delta$ with exponential rate $r=1/(1-\eta)$.
\end{theorem}
Note that $\delta_\Delta=\Omega(1/\Delta)$. 
Moreover note that this result is qualitatively similar to (but quantitatively better than) a result implicit in~\cite{lu2013improved}.
We also note that the conditions in the theorem guarantee that the graph homomorphism partition function is non-zero on graphs of maximum degree at most $\Delta$ for complex matrices $A$ satisfying $|A_{i,j}-1|\leq \delta_\Delta$ for all $i,j$ by~\cite[Theorem 7.1.4]{BarvinokBook}. 
Improvements to~\cite[Theorem 7.1.4]{BarvinokBook} do not automatically lead to improvements to Theorem~\ref{thm:main hom}, since in our proof we require that the zero-freeness also holds for graphs with boundary conditions. See Remark~\ref{rem:hom} for further discussion.

\subsection{Related work}
Our work falls into a recent series of contributions in which absence of complex zeros of the probability generating function of a discrete distribution gives rise to detailed probabilistic information about the distribution.

As mentioned earlier, the notion of strong spatial mixing is intimately connected to the design of efficient algorithms to (approximately compute) evaluations of graph polynomials and partition functions. 
Another well known approach for designing such algorithms for this task is based on Markov chains, in particular the Glauber dynamics. This of course then leads to randomized algorithms.
Very recently Chen, Liu and Vigoda~\cite{chen2021spectral}, building on~\cite{Anarietal}, showed that absence of zeros for partition functions of several models in a multivariate sense leads to proofs of rapid mixing of the Glauber dynamics for these models.

Another application of absence of complex zeros is found in~\cite{LPRS,MS1,MS2,jain2021approximate}, where central limit theorems are derived for discrete probability distributions taking a finite number of values in the nonnegative integers, whose probability generating function $p(X)$, defined as $p(x)=\sum_{k\geq 0}\Pr[X=k]x^k$, has no zeros in the vicinity of $x=1.$ 

\subsection{Overview of proof}
Our proof consists of essentially two main steps. The first step is to view the conditional probability that we try to control as an evaluation of a rational function $P(z)$ at $z=1$ and utilize absence of complex zeros to show that $|P(z)|$ is bounded on some domain containing $z=1$ and $z=0$. 
This is done in two different ways. For the graph homomorphism partition function this is done using absence of zeros in the multivariate sense, while for the independence polynomial we only require absence of zeros for the univariate polynomial by using the powerful Montel theorem from complex analysis.
Once it is known that the rational function $P(z)$ is bounded, then by using Cauchy's formula we obtain bounds on the coefficients of its series expansion. We interpret these coefficients combinatorially with the aid of the cluster expansion to arrive at the desired strong spatial mixing results.

The remainder of the paper is organized as follows. In the next section we gather the tools that we need to prove our results. In Section~\ref{sec:ind} we prove our two results on the hard-core model and in Section~\ref{sec:hom} we prove Theorem~\ref{thm:main hom}.
Finally in Section~\ref{sec:conclude} we conclude with some remarks and questions.

\section{Tools}
\subsection{Convention}
We will often deal with functions $f$ holomorphic on some open set $U\subset \mathbb{C}$ containing $0$. Therefore, near $0$, $f$ has a convergent series expansion $f(z)=\sum_{k\geq 0}a_k z^k$.
In such a situation we often just write $f(z)=\sum_{k\geq 0}a_k z^k$ near $0$. 

For $r>0$ we denote by $\mathbb{D}_r$ the open disk centered at $0$ of radius $r$.

\subsection{The cluster expansion}
The cluster expansion is a formal series expansion of the logarithm of a so-called polymer partition function~\cite{FriedliVelenik}.
The polymer partition function can also be viewed as the multivariate independence polynomial of an associated graph~\cite{ScottSokal05}, which is the perspective we take here.

Let $G=(V,E)$ be a graph. Let $w=(w_v)_{v\in V}$ be a vector of complex variables. Then the \emph{multivariate independence polynomial} of $G$ is defined as 
\begin{equation}\label{eq:def multi ind}
Z_G(w)=\sum_{\substack{I\subseteq V\\ I \text{ independent }}}\prod_{v\in I}w_v.    
\end{equation}
For a sequence of (not necessarily distinct) vertices $(v_1,\ldots,v_k)$, $v_i\in V$, $i=1,\ldots,k$, we denote by $G(v_1,\ldots,v_k)$ the graph on the vertex set $\{1,\ldots,k\}$ where for $i\neq j$,  $i$ is adjacent to $j$ if and only if $v_i=v_j$ or $\{v_i,v_j\}\in E$ and we call $G(v_1,\ldots,v_k)$ the \emph{cluster} induced by $v_1,\ldots,v_k$.
The \emph{Ursell function} of a graph $H$ is defined as 
\begin{equation}
    \phi(H)=\sum_{\substack{F\subseteq E(H)\\ (V(H),F) \text{ connected}}}(-1)^{|F|}.
\end{equation}
Note that by definition $\phi(H)=0$ if the graph $H$ is not connected.

The \emph{cluster expansion} is the following formal power series representation of $\log(Z_G(w))$~\cite{KP86,ScottSokal05},
\begin{equation}\label{eq:cluster expansion}
    \log(Z_G(w))=\sum_{k\geq 1}\frac{1}{k!}\sum_{v_1,\ldots,v_k\in V} \phi(G(v_1,\ldots,v_k)) \prod_{i=1}^k w_{v_i}.
\end{equation}
Under certain conditions on the $w_v$ the cluster expansion converges~\cite{KP86,ScottSokal05}. 
We will however not need to use these conditions.
For our purposes it suffices that if all $w_v$ are small enough in absolute value (possibly depending on the underlying graph $G$), then the cluster expansion converges.

\subsection{Some complex analysis}
The next lemma is a consequence of Cauchy's differentiation's theorem (which in turn follows from the integral formula).
\begin{lemma}\label{lem:consequence Cauchy}
Let $P(z)$ be a holomorphic function on $\mathbb{D}_r$ for some $r>1$, with series expansion $P(z)=\sum_{k\geq 0} a_k z^k$ near $0$. 
Suppose that $|P(z)|$ is bounded by $M$ on $\mathbb{D}_r$.
Then the radius of convergence of the series expansion is bigger than $1$ and for any $N\in \N$, 
\[|P(1)-\sum_{k=0}^{N-1} a_k|\leq \frac{Mr}{(r-1)r^{N}}.
\]
\end{lemma}

\begin{proof}
Choose $\rho$ so that $1<\rho<r$.
By Cauchy's differentiation's theorem we have
\[
a_k=\frac{1}{2\pi i}\int_{\partial \mathbb{D}_\rho} \frac{P(w)}{w^{k+1}} dw.
\]
This implies that $|a_k|\leq M/\rho^{k}$. Since this holds for any $1<\rho<r$, it follows that $|a_k|\leq M/r^{k}$.
Bounding $P(z)$ by a geometric series it follows that the radius of convergence is bigger than $1$.
It follows similarly that $|\sum_{k\geq N}a_k|\leq \frac{M r^{-N}}{1-1/r}=\frac{M r}{(r-1)r^{N}}$,
as desired.
\end{proof}

Typically we will not have functions defined on disks, but rather on neighbourhoods of real intervals. 

\begin{lemma}\label{lem:bound strip}
Let $P(z)=\sum_{k\geq 0} a_k z^k$ and $Q(z)=\sum_{k\geq 0} b_k z^k$ be two holomorphic functions defined on some open set containing $0$ that satisfy $a_k=b_k$ for $k=0,\ldots, N$ for some $N\in \mathbb{N}$.
Then 
\begin{itemize}
    \item[(i)] If there exists $\eps>0$ and $M>0$ such that both $|P(z)|$ and $|Q(z)|$ are bounded by $M$ on $\mathcal{N}([0,1],2\eps)$, then there exists a constant $r=1+O(e^{-1/\eps})$ such that
\begin{equation}\label{eq:bound 1}
    |P(1)-Q(1)|\leq \frac{2M r}{(r-1)r^{N}}.
\end{equation}
\item[(ii)] If there exists $\delta>0$ such that for any compact set $S\subset \{z\mid \Re(z)>-\delta\}$ intersecting the positive real line there exists a constant $M=M_S$ such that both $|P(z)|$ and $|Q(z)|$ are bounded by $M_S$ on $S$, then there exists a compact set $S$ such that with $r= 1+\sqrt{\delta}$,
\begin{equation}\label{eq:bound 2}
    |P(1)-Q(1)|\leq \frac{2M_S r}{(r-1)r^{N}}.
\end{equation}
\end{itemize}
\end{lemma}
\begin{proof}
We start with the proof of part (i).
Let $r=\frac{1-e^{-1-1/\eps}}{1-e^{-1/\eps}}$ and let $\alpha=1-e^{-1/\eps}$. Note that $r=1+O(e^{-1/\eps})$. 
Define $g(z)=\eps \log(1/(1-\alpha z))$ on $\mathbb{D}_r$ taking the branch of the logarithm that satisfies $g(0)=0$. Barvinok shows in his proof of~\cite[Lemma 2.2.3]{BarvinokBook} that $g$ maps the disk $\mathbb{D}_r$ into $\mathcal{N}([0,1],2\eps)$ and that $g(1)=1.$

We now consider the compositions $P\circ g$ and $Q\circ g$ on $\mathbb{D}_r$.
Let us write $P\circ g=\sum_{k\geq 0} a_k g(z)^k=\sum_{k\geq 0} a'_k z^k$ and $Q\circ g=\sum_{k\geq 0} b_k g(z)^k=\sum_{k\geq 0}b'_k z^k$.
Then the coefficients $a'_k$ (resp. $b'_k$) depend only on $a_0,\ldots, a_k$ (resp. $b_0,\ldots, b_k)$ and the first $k$ coefficients of the Taylor series of $g$ around $0$, since $g$ has constant term equal to $0$.
This implies that $a'_k=b'_k$ for $k=0,\ldots, N.$
Since both $|(P\circ g)(z)|$ and $|(Q\circ g)(z)|$ are bounded by $M$ on $\mathbb{D}_r$, Lemma \ref{lem:consequence Cauchy} in combination with the triangle inequality now implies that 
\[
|P(1)-Q(1)|=|(P\circ g)(1)-(Q\circ g) (1)|\leq 2 \frac{M r}{(r-1)r^{N}},
\]
as desired.

For the proof of part (ii) consider for $\xi=1-\sqrt{\frac{\delta}{1+\delta}}$ the map $h(z)=\frac{\delta}{(1-\zeta z)^2}-\delta$. By Lemma 2.4 of~\cite{Barvinokreal}, $h$ maps the open disk $\mathbb{D}_{\xi^{-1}}$ into the set $\mathbb{C}\setminus \{z\in \mathbb{R}\mid z<-3/4\delta\}$ and satisfies $h(0)=0$ and $h(1)=1.$
Let $r=1+\sqrt{\delta}$. Then $1<r<\xi^{-1}.$
Denote by $S$ the image under $h$ of the closure of $\mathbb{D}_{r}$. Then $S$ is a compact set contained in $\{z\mid \Re(z)\geq -\delta\}$ and therefore both $P$ and $Q$ are bounded on this set by some constant $M:=M_S=M_\delta$.
The proof now proceeds in exactly the same way as in case (i).
\end{proof}
\begin{remark}
We could have also used the Riemann mapping theorem in the proof above to get suitable map $g$ and $h$. For concreteness the present ones are convenient, as we can easily compute their Taylor series, thereby making them suitable for algorithmic applications.
\end{remark}

\subsubsection*{Montel's theorem}
An important tool in our proof of Theorem~\ref{thm:main claw} is the use of Montel's theorem. 
This is a cornerstone result in the theory of modern complex dynamical systems~\cite{zalcman,complexdynamicsbook,Milnor} and has recently found applications in the study of determining the location of zeros and the complexity of approximating evaluations of the independence polynomial~\cite{BGGS,buys,chaoticratios}.

We need a definition to state the theorem.
We denote by $\widehat \C=\C\cup \{\infty\}$ the extended complex plane.
Let $U\subset\C$ be an open set. A family $\mathcal{F}$ of holomorphic functions $f:U\to \widehat \C$ is called a \emph{normal} family if each infinite sequence of elements of $\mathcal{F}$ has a subsequence that converges locally uniformly to a holomorphic function.

\begin{theorem}[Montel]\label{thm:montel}
$U\subset\C$ be a connected open set and let $\mathcal{F}$ be a family of holomorphic functions $f:U\to \widehat \C$.
Suppose that there exists three distinct points $a,b,c\in \widehat\C$ such that $f(U)\subset \widehat \C \setminus \{a,b,c\}$ for all $f\in \mathcal{F}$. Then the family $\mathcal{F}$ is normal.
\end{theorem}
See e.g.~\cite{zalcman,complexdynamicsbook,Milnor} for variations, extensions and a proof of Theorem~\ref{thm:montel}.

\section{The hard-core model}\label{sec:ind}
In this section we will prove Theorems~\ref{thm:main general} and ~\ref{thm:main claw}.

Let us introduce for a graph $G$ and a vertex $v$ of $G$, the \emph{ratio},
\begin{equation}\label{eq:def ratio}
P_{G,v}(\lambda)=\frac{\lambda Z_{G\setminus N[v]}(\lambda)}{Z_{G}(\lambda)},
\end{equation}
considered as a rational function in $\lambda$. Here $N[v]$ denotes the closed neighbourhood of $v$.
Note that for positive $\lambda$, $P_{G,v}(\lambda)$ is just the probability of the vertex $v$ being in the random independent set drawn from the hard-core measure. 


We introduce some further notation to facilitate the discussion.
Let $G=(V,E)$ be a graph and let $v\in V$ and $\Lambda\subset V\setminus \{v\}$. For a boundary condition $\sigma$ on $\Lambda$ we denote by $G[\sigma]$ be the graph obtained from $G\setminus \Lambda$ by removing all neighbours in $G\setminus \Lambda$ of vertices in $\Lambda$ that are set to `in' by $\sigma$. Note that $G[\sigma]$ is an induced subgraph of $G$.
Denote by $\sigma_{v,1}$ the boundary condition on $\Lambda\cup \{v\}$ extending $\sigma$ that assigns $1$ to $v$.
Let $\lambda>0$. Then
\begin{equation}\label{eq:cond prob is ratio}
\Pr_{\mu}[v \in I\mid \sigma]=P_{G[\sigma],v}(\lambda).
\end{equation}
Indeed, by Bayes' rule, writing $\Pr_\mu[\sigma]$ for the probability that the random independent set $I$ drawn according to $\mu$ satisfies $\sigma^{-1}(1)\subseteq I$ and $\sigma^{-1}(0)\cap I=\emptyset$, we have
\[
Pr_{\mu}[v \in I\mid \sigma]=\frac{\Pr_\mu[\sigma_{v,1}]}{\Pr_\mu[\sigma]}=\frac{\lambda^{|\sigma^{-1}(1)|+1}Z_{G[\sigma_{v,1}]}(\lambda)/Z_G(\lambda)}{\lambda ^{|\sigma^{-1}(1)|}Z_{G[\sigma]}(\lambda)/Z_G(\lambda)}=P_{G[\sigma],v}(\lambda),
\]
as claimed.

\subsection{Bounded ratios imply strong spatial mixing}
We start by giving a series expansion of the ratios.
For a sequence of (not necessarily distinct) vertices $(v_1,\ldots,v_k)$ from $G$ and a vertex $v\in V$ we denote by $m_v(v_1,\ldots,v_k)$ the number of $i$ such that $v=v_i$.
 \begin{lemma}\label{lem:ratio series}
Let $G=(V,E)$ be a graph, $v$ a fixed vertex of $G$ and let $\lambda$ be a complex variable. Then near $\lambda=0$ we have the following series expansion of $P_{G,v}$
\begin{equation}
P_{G,v}(\lambda)=\sum_{k\geq 1}\frac{1}{k!} \sum_{v_1,\ldots,v_{k}\in V}\phi(G(v_1,\ldots,v_k))m_v(v_1,\ldots,v_{k})\lambda^{k}.    
\end{equation}
\end{lemma}
\begin{proof}
We introduce a vector of complex variables $w=(w_v)_{v\in V}$.
Then for the multivariate independence polynomial, we have 
\[
\frac{w_{v}Z_{G\setminus N[v]}(w)}{Z_G(w)}=w_v\frac{\partial}{\partial w_{v}} \log(Z_G(w)),
\]
as long as $Z_G(w)\neq 0$.
By the cluster expansion~\eqref{eq:cluster expansion}, this implies that for small enough $w_u$ (for all $u\in V$), we have
\begin{align*}
\frac{w_{v}Z_{G\setminus N[v]}(w)}{Z_G(w)}&=w_v\sum_{k\geq 1}\frac{1}{k!}\sum_{v_1,\ldots,v_{k}\in V} \phi(G(v_1,\ldots,v_{k}))\frac{\partial}{\partial w_{v}}\prod_{i=1}^{k} w_{v_i}.
\end{align*}
Now note that
\[
\frac{\partial}{\partial w_{v}}\prod_{i=1}^{k} w_{v_i}=m_v(v_1,\ldots,v_k)\frac{\prod_{i=1}^{k} w_{v_i}}{w_v}.
\]
By plugging in $w_v=\lambda$ for each $v$ and making use of the uniqueness of the coefficients of the power series representation, this completes the proof.
\end{proof}
For a graph $G=(V,E)$, a vertex $v\in V$ and a positive integer $k$, we denote
\[B_G(v,k):=\{u\in V\mid d_G(u,v)\leq k\},
\]
the \emph{distance at most $k$-neighbourhood of $v$}.
\begin{lemma}\label{lem:bounded implies SSM}
Let $\mathcal{G}$ be a family of graphs that is closed under taking induced subgraphs.
\begin{itemize}
    \item[(i)] Let $\lambda^\star$ be such that there exists constants $\eps>0$ and $M>0$ such that for all $\lambda\in \mathcal{N}([0,\lambda^\star],\eps)$, the ratios, $P_{G,v}$, satisfy $|P_{G,v}(\lambda)|\leq M$ for all $G\in \mathcal{G}$ and all vertices $v\in V(G)$. 
Then for any $\lambda\in (0,\lambda^\star]$ the hard-core measure satisfies strong spatial mixing on $\mathcal{G}$ with exponential rate $r=1+O(e^{-\lambda/\eps})$.
\item[(ii)] If there exists $\delta>0$ such that for any compact set $S\subset \{z\mid \Re(z)>-\delta\}$ intersecting the positive real line there exists a constant $M=M_S$ such that the ratios, $P_{G,v}$, satisfy $|P_{G,v}(\lambda)|\leq M_S$ for all $G\in \mathcal{G}$ all vertices $v\in V(G)$ and all $\lambda\in S$. 
Then for any $\lambda>0$ the hard-core measure satisfies strong spatial mixing on $\mathcal{G}$ with exponential rate $r=1+O(\sqrt{\delta/\lambda})$.
\end{itemize}
\end{lemma}

\begin{proof}

We start with the proof of (i).
Fix $\lambda\in [0,\lambda^\star]$ and let $\eps'>0$ be such that $\lambda\cdot \mathcal{N}([0,1],\eps')\subset \mathcal{N}([0,\lambda^\star],\eps))$.
Note that $\eps'=O(\eps/\lambda)$.
Let $r>1$ be the constant obtained from Lemma~\ref{lem:bound strip}(i) upon input of $\eps'/2$.
Choose any $G\in \mathcal{G}$ and fix a vertex $v$ of $G$. Choose $\Lambda\subset V(G)\setminus \{v\}$ and two boundary conditions $\sigma=\sigma_\Lambda$ and $\tau=\tau_\Lambda$ on $\Lambda$. 
We will show that 
\begin{equation}\label{eq:prob bound}
   \big|\prob_{G,\lambda}[v \text{ in }\mid \sigma_\Lambda]-\prob_{G,\lambda}[v \text{ in } \mid \tau_\Lambda]\big|\leq \frac{2Mr}{(r-1)r^{d_G(v,\sigma\neq \tau)-1}},
\end{equation}
implying the desired statement.

Define for $z\in \C$, rational functions in $z$, $P(z):=P_{G[\sigma],v}(z\lambda)$ and $Q(z):=P_{G[\tau],v}(z \lambda)$.
Then by~\eqref{eq:cond prob is ratio},
\begin{equation}\label{eq:scaled ratios at 1}
\prob[v \text{ in }\mid \sigma_\Lambda]=P(1) \quad \text{and} \quad \prob[v \text{ in }\mid \tau_\Lambda]=Q(1).
\end{equation}
By Lemma~\ref{lem:ratio series} we know that, as formal series in $z$,
\begin{align}\label{eq:series in z}
P(z)&=\sum_{k\geq 1}\frac{1}{k!} \sum_{v_1,\ldots,v_{k}\in V}\phi(G[\sigma](v_1,\ldots,v_{k}))m_v(v_1,\ldots,v_k)\lambda^{k}z^{k},\nonumber
\\
Q(z)&=\sum_{k\geq 1}\frac{1}{k!} \sum_{v_1,\ldots,v_{k}\in V}\phi(G[\tau](v_1,\ldots,v_{k}))m_v(v_1,\ldots,v_k)\lambda^{k}z^{k}.   
\end{align}
In particular the coefficient of $z^{k}$ in $P$ (resp. $Q$) depends only on the vertices in $G[\sigma]$ (resp. $G[\tau]$) that have distance at most $k-1$ to $v$ since by definition $\phi(G(v_1,\ldots,v_k))m_v(v_1,\ldots,v_k)=0$ if the cluster induced by $v_1,\ldots,v_k$ is not connected or if it does not contain the vertex $v$.
For $k<d_G(v,\sigma\neq \tau)-1$ the distance at most $k$-neighbourhoods, $B_{G[\sigma]}(v,k)$ and $B_{G[\tau]}(v,k)$ are equal.
Therefore for any $k=0,\ldots,d_G(v,\sigma\neq \tau)-1$ we know that the coefficients of $z^k$ of the power series representations for $P(z)$ and $Q(z)$ are equal.

Now by construction for any $z\in \mathcal{N}([0,1],\eps')$ we have $\lambda z\in \mathcal{N}([0,\lambda^\star],\eps)$ and therefore $|P(z)|$ and $|Q(z)|$ are bounded by $M$ on $\mathcal{N}([0,1],\eps')$.
We now use Lemma~\ref{lem:bound strip}(i) to conclude that $|P(1)-Q(1)|$ is bounded by 
\[\frac{2Mr}{(r-1)r^{d_G(v,\sigma\neq \tau)-1}},\] 
proving~\eqref{eq:prob bound}.

The proof of (ii) is very similar.
Again we fix $\lambda>0$ and define $\delta'=\delta/\lambda$. Let $r>1$ be the constant obtained from Lemma~\ref{lem:bound strip}(ii) upon input of $\delta'$.
As in the proof of (i), choose any $G\in \mathcal{G}$ and fix a vertex $v$ of $G$. 
Choose $\Lambda\subset V(G)\setminus \{v\}$ and two boundary conditions $\sigma=\sigma_\Lambda$ and $\tau=\tau_\Lambda$ on $\Lambda$. 
As above we define $P(z)=P_{G[\sigma],v}(\lambda z)$ and $Q(Z)=P_{G[\tau],v}(\lambda z)$. 
We let $S$ be the compact set from Lemma~\ref{lem:bound strip}(ii) upon input $\delta'$ and let $M=M_S$.
Then in exactly the same way as above, with $r=1+\sqrt{\delta'}$ we obtain~\eqref{eq:prob bound}.
This finishes the proof.
\end{proof}

\subsubsection*{Algorithm}
To approximately compute the conditional probabilities $\Pr_\mu[v\in {I}\mid \sigma]$ we need to approximate $P$ evaluated at $1$.
The basic idea is just to truncate the series for $P\circ g$ at depth $K=O(\log(n/\eps))$ for an $n$-vertex graph $G$, so as to obtain an additive $\eps/n$-approximation.
Since the coefficients of $P\circ g$ can be easily computed from those of $g$ and $P$ (in time $O(K^2)$ using Horner's method), the real algorithmic task is to compute the first $K$ coefficients of $P$. This can be done efficiently (i.e. in time $\Delta^{O(K)}$ on graphs of maximum degree at most $\Delta$) with an algorithm appearing in the proof of Theorem 6 from~\cite{TPR20}.
We leave the details to the reader.

\subsection{Absence of zeros implies bounded ratios}
In this section we cover the final ingredients to prove our main result for the hard-core measure on bounded degree graphs.

Before we start we note that for a vertex transitive graph $G$ it is not hard to see that \[\frac{P_{G,v}(\lambda)}{\lambda}=\frac{d}{d\lambda}\frac{\log(Z_G(\lambda))}{|V(G)|}\] 
and therefore absence of zeros of $Z_G$ on an infinite family of vertex transitive graphs on some open set containing $0$ implies boundedness of the ratios, $P_{G,v}$, as can be derived from the proof of~\cite[Lemma 2.2.1]{BarvinokBook} (in combination with the Riemann mapping theorem). 
Below we show this is also true for bounded degree graphs in general.
\begin{lemma}\label{lem:zero-free implies bounded ratios}
Let $\mathcal{G}$ be a family of graphs that is closed under taking induced subgraphs.
Let $U\subseteq \C$ be an open set such that for any graph $G\in \mathcal{G}$ and any $\lambda\in U$ $Z_G(\lambda)\neq 0$.
Then for any compact set $S\subset U\setminus \{0\}$ that intersects the positive real line there exists a constant $M>0$ such that for all $G\in \mathcal{G}$, $v\in V(G)$ and $\lambda$ in $S$,
\[
|P_{G,v}(\lambda)|\leq M.
\]
\end{lemma}
\begin{proof}
Suppose to the contrary that the ratios are unbounded on $S$.
Then there exists a sequence of graphs $(G_n)_{n\geq 1}$ and vertices $v_n\in V(G_n)$ and a sequence of points $(\lambda_n)_{n\geq 1}$ in $S$ such that
\begin{equation}\label{eq:unbounded ratios sequence}
|P_{G_n,v_n}(\lambda_n)|\geq n.
\end{equation}
Since for any graph $G$ and vertex $v\in V(G)$, $Z_{G}(\lambda)=\lambda Z_{G\setminus N[v]}(\lambda)+Z_{G-v}(\lambda)$ 
and since $\mathcal{G}$ is closed under taking induced subgraphs, it follows that for any $\lambda\in U\setminus \{0\}$, the ratio $P_{G,v}(\lambda)$ must avoid the points $\infty,0$ and $1$.
Therefore, by Theorem~\ref{thm:montel} (Montel's theorem), the family of rational functions $\{\lambda\mapsto P_{G_n,v_n}(\lambda)\mid n\geq 1 \}$ forms a normal family on $U\setminus \{0\}$ and hence contains a subsequence that converges locally uniformly to some holomorphic function $f$.
In particular this convergence is uniform on $S$ by compactness.
Since for positive real $\lambda$ the ratios are just probabilities and hence contained in $[0,1]$, it follows that $f$ is not constant $\infty$.
But then the image of $S$ under $f$ must be bounded, implying that the ratios $P_{G_n,v_n}(\lambda_n)$ cannot be unbounded on $S$ by uniform convergence.
This contradicts our assumption and therefore there must be some bound $M=M_S$ on the absolute values of the ratios on the set $S$, as desired.
\end{proof}

The next lemma shows boundedness of the ratios for $\lambda$ near $0$ and is essentially due to Shearer~\cite{shearer} and Scott and Sokal~\cite{ScottSokal05}.
\begin{lemma}\label{lem:bounded ratios near zero}
Let $\mathcal{G}_\Delta$ be the family of graphs of maximum degree at most $\Delta\geq 3$. For any $\lambda$ such that $|\lambda|< \frac{(\Delta-1)^{\Delta-1}}{\Delta^\Delta}$, any graph $G\in \mathcal{G}_\Delta$, $v\in V(G)$ we have $|P_{G,v}(\lambda)|< \frac{1}{\Delta-2}$ and moreover $Z_{G}(\lambda)\neq 0$.
\end{lemma}
\begin{proof}
We use~\cite[Lemma 2.10]{chaoticratios} that states that for $\lambda$ as in the statement of the lemma we have with $R_{G,v}(\lambda):=\frac{\lambda Z_{G\setminus N[v]}(\lambda)}{Z_{G-v(\lambda)}}$, $|R_{G,v}(\lambda)|<\frac{1}{\Delta-1}$ and $Z_{G}(\lambda)\neq 0$.
Using that $P_{G,\lambda}=\frac{R_{G,v}(\lambda)}{1+R_{G,v}(\lambda)}$ the first statement also follows.
\end{proof}

We can now prove our main results concerning strong spatial mixing for the hard-core measure. 
\begin{proof}[Proof of Theorem~\ref{thm:main general}]
Let $\mathcal{G}\subset \mathcal{G}_{\Delta}$ be a family of graphs of maximum degree at most $\Delta$ that is closed under taking induced subgraphs and let $\lambda^\star>0$ and $\eps>0$ be such that for all $\lambda\in \mathcal{N}([0,\lambda^\star],\eps)$ and $G\in \mathcal{G}$, $Z_G(\lambda)\neq 0.$
Then by Lemma~\ref{lem:bounded ratios near zero} and Lemma~\ref{lem:zero-free implies bounded ratios} we know that the ratios $P_{G,v}$ are bounded on $\mathcal{N}([0,\lambda^\star],\eps)$ by some constant $M$ for all graphs $G\in \mathcal{G}$ and all $v\in V(G)$.
Therefore the result follows from Lemma~\ref{lem:bounded implies SSM}(i).
\end{proof}

\begin{proof}[Proof of Theorem~\ref{thm:main claw}]
Let $\mathcal{G}'_\Delta\subset \mathcal{G}_{\Delta}$ be the family of claw-free graphs of maximum degree at most $\Delta$.
By the Chudnovsky-Seymour theorem~\cite{ChudnovskySeymour07} we know that all roots of $Z_G$ for $G\in \mathcal{G}'_\Delta$ are negative reals.
By Lemma~\ref{lem:bounded ratios near zero}, it follows that $Z_{G}(\lambda)\neq 0$ as long as $\lambda>-\tfrac{1}{e\Delta}$.
Combining Lemma~\ref{lem:bounded ratios near zero} and Lemma~\ref{lem:zero-free implies bounded ratios} we conclude that on any compact set $S$ that avoids the set $\{z\in \mathbb{R}\mid z\leq -\tfrac{1}{e\Delta}\}$ the ratios, $P_{G,v}$, for $G\in \mathcal{G}'_\Delta$ and any $v\in V(G)$ are bounded in absolute value by some constant $M_S$.
Therefore, Lemma~\ref{lem:bounded implies SSM}(ii) implies that hard-core measure on $\mathcal{G}'_\Delta$ at any $\lambda>0$ satisfies strong spatial mixing with exponential rate $r=1+O((\lambda \Delta)^{-1/2})$.
\end{proof}
\begin{remark}
Bencs~\cite{Bencs18} shows that the independence polynomials of graphs containing a moderate number of claws are zero free in some sector. 
In a similar way one can show that also for bounded degree graphs in this class of graphs the hard-core measure satisfies strong spatial mixing cf.~\cite[Section 3.5]{Barvinokreal}.
\end{remark}

\section{The graph homomorphism partition function}\label{sec:hom}
Here we prove Theorem~\ref{thm:main hom}. We follow the same strategy as in the previous section.

Let us start by introducing the ratios. 
Let $G=(V,E)$ be a graph and let $v\in V$, $i\in [q]$ and let $A$ be a symmetric, nonnegative $q\times q$ matrix, where $q$ is a positive integer such that $Z_G(A)>0$.
For a boundary condition $\sigma:\Lambda\to [q]$ on some set $\Lambda\subset V\setminus \{v\}$. 
We denote by $\sigma_{v,i}$ the extension of $\sigma$ to $\Lambda\cup \{v\}$ that assigns $i$ to the vertex $v$. 
In what follows we denote by $J$ the all ones matrix of the appropriate size.
We define the following rational function in the variable $z$
\begin{equation}\label{eq:ratio hom}
    P^\sigma_{G,v,i;A}(z)=\frac{Z^{\sigma_{v,i}}_{G}(J+z(A-J))}{Z^{\sigma}_{G}(J+z(A-J))}
\end{equation}
and refer to it as a \emph{ratio at $v$}.
We note that, as in the case of the hard-core model, we have 
\[
P^\sigma_{G,v,i;A}(1)=\Pr_{\mu_A}[\phi(v)=i\mid \sigma].
\]

\subsection{Ratios and their series expansion}
We do have to do a bit more work to find the series expansion of these ratios. 
In particular, we need to equip the model with an external field parameter $\xi\in \C^{V\times [q]}$.
We define for a graph $G=(V,E)$,
\begin{equation}\label{eq:graph hom def external}
  Z_{G}(A,\xi)=\sum_{\phi:V\to [q]}\prod_{v\in V}\xi_{v,\phi(v)}\cdot \prod_{uv\in E}A_{\phi(u),\phi(v)}.  
\end{equation}
Moreover, for a boundary condition $\sigma:\Lambda\to [q]$ on some set $\Lambda\subset V$ we denote by $Z^\sigma_G(A,\xi)$ the partition function defined as above where we only sum over those $\phi$ that restricted to $\Lambda$ coincide with $\sigma$.
The following lemma explains the usefulness of introducing the external field parameters.
\begin{lemma}\label{lem:ratio is derivative}
Let $A$ be a symmetric $q\times q$ matrix. Let $G=(V,E)$ be a graph, let $v\in V$ and $i\in [q]$ and let $\Lambda\subset V\setminus \{v\}$ be equipped with a boundary condition $\sigma: \Lambda\to [q]$.
Then
\[
P^\sigma_{G,v,i;A}(z)=\frac{\partial}{\partial \xi_{v,i}}\log(Z_{G}(J+z(A-J),\xi))|_{\xi=1}.
\]
\end{lemma}
\begin{proof}
This follows directly from the fact that 
\[
\frac{\partial}{\partial \xi_{v,i}}Z^{\sigma}_{G}(J+z(A-J),\xi)|_{\xi=1}=Z^{\sigma_{v,i}}_{G}(J+z(A-J))
\]
and the standard rules for the derivative of the logarithm.
\end{proof}

Next we wish to use the cluster expansion to find a series expansion for $\log(Z^\sigma_{G}(J+z(A-J),\xi))$ with $\sigma$ a boundary condition on some set $\Lambda\subseteq V\setminus\{v\}$.
To do this we will have to realize the graph homomorphism partition function as the multivariate independence polynomial of an auxiliary graph $\Gamma$. This will be done in a similar way as in~\cite{sokalchrom,borgsetal}.

The vertex set of the auxiliary graph $\Gamma$ will consist of the connected subgraphs of $V(G)$ with at least one edge. Two vertices $H_1=(S_1,E_1)$ and $H_2=(S_2,E_2)$ of $\Gamma$ are connected by an edge if and only if $S_1$ and $S_2$ intersect.
Next we define the vertex weights. For a connected subgraph $H=(S,F)$ of $G$ we define the weight, $w^\sigma(H)$, of $H$ by
\begin{equation}\label{eq:weight graph hom}
w^\sigma(H):=\frac{z^{|F|}Z^{\sigma}_H(A-J,\xi)}{\left(\prod_{v\in S\setminus \Lambda}\sum_{i=1}^q\xi_{v,i}\right)\cdot \prod_{v\in \Lambda\cap S}\xi_{v,\sigma(v)}},
\end{equation}
where we understand $\sigma$ to be restricted to $\Lambda\cap S$.
We also define
\[
p^\sigma(\xi):=\left(\prod_{v\in V\setminus \Lambda}\sum_{i=1}^q\xi_{v,i}\right)\cdot \prod_{v\in \Lambda}\xi_{v,\sigma(v)}.
\]
\begin{lemma}\label{lem:hom is hard-core}
With definitions as above we have
\[
 p^\sigma(\xi)Z_{\Gamma}(w^\sigma)=Z^\sigma_{G}(J+z(A-J),\xi).
\]
\end{lemma}
\begin{proof}
This follows from expanding the product over $E$ in the definition of $Z^\sigma_{G}(J+z(A-J),\xi)$. 
We have that  $(p^\sigma(\xi))^{-1}Z^\sigma_{G}(J+z(A-J),\xi)$ is equal to
 \\
\begin{align*}
&(p^\sigma(\xi))^{-1}\sum_{\substack{\phi:V\to [q]\\\phi|_{\Lambda}=\sigma}}\prod_{v\in V}\xi_{v,\phi(v)}\cdot \prod_{uv\in E}( J+z(A-J))_{\phi(u),\phi(v)}
    \\
    =&(p^\sigma(\xi))^{-1} \sum_{\substack{\phi:V\to [q]\\ \phi|_{\Lambda}=\sigma}}\prod_{v\in V}\xi_{v,\phi(v)}\cdot \sum_{F\subseteq E}z^{|F|}\prod_{uv\in F}(A-J)_{\phi(u),\phi(v)}
    \\
    =&\sum_{F\subseteq E} z^{|F|} \left(\prod_{v\in V(F)\setminus \Lambda}\sum_{i=1}^q\xi_{v,i}\right)^{-1}\cdot \left(\prod_{v\in \Lambda\cap V(F)}\xi_{v,\sigma(v)}\right)^{-1} \times 
    \\
    \quad \quad \quad \quad \quad  &\sum_{\substack{\phi:V(F)\to [q]\\\phi|_{\Lambda\cap V(F)}=\sigma|_{\Lambda\cap V(F)}}}\prod_{v\in V(F)}\xi_{\phi(v)}\cdot \prod_{uv\in F}(A-J)_{\phi(u),\phi(v)}.
\end{align*}
Now for $F\subseteq E$ fixed, the contribution to the sum is multiplicative over the connected components of $F$ and for such a component $H$ this contribution is exactly given by $w^\sigma(H)$.
This implies the statement of the lemma as the independent sets in $\Gamma$ are exactly the collections of pairwise vertex disjoint connected subgraphs with at least one edge of $G$.
\end{proof}

For a graph $G$, two positive integers $\ell,k$ and a vertex $v\in V(G)$ we define $\mathcal{C}_{v;\ell,k}(G)$ to be the collection consisting of sequences $(H_1,\ldots,H_k)$ of connected subgraphs of $G$ with at least two vertices satisfying
\begin{itemize}
    \item[(i)] $\sum_{j=1}^k |E(H_j)|=\ell$,
    \item[(ii)] $v\in \bigcup_{i=j}^k V(H_j)$,
    \item[(iii)] the graph $\Gamma(H_1,\ldots,H_k)$ is connected.
\end{itemize}
Let us denote the scaled weights $\widehat{w}^\sigma(H)=w^\sigma(H)z^{-|E(H)|}$ for any connected subgraph $H$ of $G$.
By applying the cluster expansion to $Z_\Gamma(w^\sigma)$ we obtain the following series expansion for the ratio:

\begin{lemma}\label{lem:series ratio hom}
As a series in $z$ we have that $P^\sigma_{G,v,i;A}(z)$ near $z=0$ is equal to 
\begin{equation}
    1/q+\sum_{\ell\geq 1}z^\ell\sum_{k\geq 1}\frac{1}{k!}\sum_{(H_1,\ldots,H_k)\in \mathcal{C}_{v;\ell,k}(G)}\phi(\Gamma(H_1,\ldots,H_k)) \frac{\partial}{\partial \xi_{v,i}}\prod_{j=1}^k w^\sigma(H_j)|_{\xi=1}.
\end{equation}
In particular, the $\ell$-th term of the series only depends on the distance at most $\ell$ neighbouhood of the vertex $v$ in $G$ (and the boundary condition $\sigma$ restricted to this neighbourhood).
\end{lemma}
\begin{proof}
By Lemma~\ref{lem:ratio is derivative} and the previous lemma it suffices to compute the partial derivative with respect to $\xi_{v,i}$ of $\log(p^\sigma(\xi))$ and $\log(Z_\Gamma(w^\sigma))$, evaluate the result at $\xi=1$ and add these.

It is not difficult to see that
\[
\frac{\partial}{\partial\xi_{v,i}}(\log(p^\sigma(\xi)))|_{\xi=1}=1/q.
\]

For the other derivative we first use~\eqref{eq:cluster expansion} to obtain that as a series in $z$, near $z=0$, whenever the $\xi_{u,j}$ are sufficently close to $1$, we have
\begin{align*}
\log(Z_\Gamma(w^\sigma))&=\sum_{k\geq 1}\frac{1}{k!}\sum_{H_1,\ldots,H_k\in V(\Gamma)} \phi(\Gamma(H_1,\ldots,H_k)) \prod_{i=1}^k
w^\sigma(H_i).
\\
&=\sum_{\ell\geq 1}z^\ell\sum_{k\geq 1}\frac{1}{k!}\sum_{\substack{H_1,\ldots,H_k\in V(\Gamma)\\ \sum_{i=1}^q |E(H_i)|=\ell}} \phi(\Gamma(H_1,\ldots,H_k)) \prod_{i=1}^k \widehat{w}^\sigma(H_i).
\end{align*}
Next observe that for a connected subgraph $H$ of $G$ we have $\frac{\partial}{\partial \xi_{v,i}}\widehat{w}^\sigma(H)=0$ if $v\notin V(H)$.
Therefore $\frac{\partial}{\partial \xi_{v,i}}(\log(Z_\Gamma(w^\sigma))|_{\xi=1}$ has the following series expansion in $z$ near $z=0:$
\[
\sum_{\ell\geq 1}z^\ell\sum_{k\geq 1}\frac{1}{k!}\sum_{(H_1,\ldots,H_k)\in \mathcal{C}_{v;\ell,k}(G)}\phi(\Gamma(H_1,\ldots,H_k)) \frac{\partial}{\partial \xi_{v,i}}\prod_{j=1}^k \widehat{w}^\sigma(H_j)|_{\xi=1}.
\]
This finishes the proof.
\end{proof}

\subsection{Bounded ratios imply strong spatial mixing}
\begin{lemma}\label{lem:bounded implies SSM hom}
Let $q\geq 2$ be an integer and let $A$ be a nonnegative and nonzero symmetric $q\times q$ matrix.
Let $\mathcal{G}$ be a family of graphs.
Suppose there exists constants $r>1$ and $M>0$ such that for all $z\in \mathbb{D}_{r}$, the ratios, $P^\sigma_{G,v,i;A}$, satisfy $|P^\sigma_{G,v,i;A}(z)|\leq M$ for all $G\in \mathcal{G}$, all vertices $v\in V(G)$ and all boundary conditions $\sigma:\Lambda\to[q]$ for $\Lambda\subset V\setminus\{v\}.$ 
Then the measure $\mu_A$ satisfies strong spatial mixing on $\mathcal{G}$ with exponential rate $r$.
\end{lemma}
\begin{proof}
Let $G=(V,E)\in \mathcal{G}$, let $v\in V$ and let $\sigma$ and $\tau$ be two boundary conditions on some set $\Lambda\subset V\setminus \{v\}$.
Since by Lemma~\ref{lem:series ratio hom} for $\ell=0,\ldots d_G(\sigma\neq \tau)-1$ the coefficients of the series expansion around $z=0$ for $P^\sigma_{G,v,i;A}(z)$ and $P^\tau_{G,v,i;A}(z)$ are the same, the result follows from applying Lemma~\ref{lem:consequence Cauchy} to $P(z):=P^\sigma_{G,v,i;A}(z)-P^\tau_{G,v,i;A}(z)$.
\end{proof}

\subsubsection*{Algorithms}
Just as for the hard-core model, implicit in our proof of the above lemma there is an efficient algorithm for (approximately) computing the probabilities $\Pr_\mu[{\psi}(v)=i\mid \sigma]$ on graphs of maximum degree at most $\Delta$.
Again the basic idea is just to truncate the series for $P^\sigma_{G,v,i;A}$ at convenient depth, say $K$, so as to obtain the desired approximation.
This can be done efficiently (i.e. in time $\Delta^{O(K)}$) with an algorithm appearing in the proof of Theorem 6 from~\cite{TPR20}.
We leave the details to the reader.

\subsection{Bounded ratios from absence of zeros}
For a graph $G=(V,E)$ with a given orientation of the edges and $q\times q$ matrices $A^e$ for each edge $e=(u,v)$, we define 
\[
Z_{G}((A^e)_{e\in E})=\sum_{\psi:V\to [q]}\prod_{(u,v)\in E}A^{(u,v)}_{\psi(u),\psi(v)}.
\]
As before we have a similar definition for $Z^\sigma_{G}((A^e)_{e\in E})$ for a boundary condition $\sigma$ on some $\Lambda\subseteq V$.

We will need the following result due to Barvinok~\cite{BarvinokBook}:
\begin{theorem}[Theorem 7.1.4 of \cite{BarvinokBook}]\label{thm:zero-free hom barvinok}
Let $\Delta\geq 3$ and $q\geq 2$ be integers and let for some $\alpha=\alpha_\Delta<2\pi/3\Delta$,
\[
\delta_\Delta=\sin(\alpha/2)\cos(\alpha\Delta/2).
\]
let $A$ be any $q\times q$ matrix satisfying $|A_{i,j}-1|\leq \delta_\Delta$ for all $i,j=1,\ldots,q$.
Then for any orientation of any graph $G=(V,E)\in \mathcal{G}_\Delta$ and any boundary condition $\sigma$ on any $\Lambda\subseteq V$, $Z^\sigma_G(A)\neq 0$.
\end{theorem}
\begin{remark}
In fact this theorem is not stated as Theorem 7.1.4 in \cite{BarvinokBook}.
However, in his proof of Theorem 7.1.4 in~\cite{BarvinokBook}, Barvinok shows that the statement of the theorem is true for \emph{symmetric} matrices and ordinary graphs satisfying the condition. The extension to not necessarily symmetric matrices follows along exactly the same lines. See~\cite{chen2021spectral} for a proof of an analogues statement derived from~\cite{BarvinokSoberon}.   
We therefore omit a proof.
\end{remark}
\begin{remark}\label{rem:hom}
Barvinok has proven a stronger zero-freeness result for matrices whose entries are close to the real axis~\cite[Theorem 7.2.2]{BarvinokBook}. 
However it only applies to boundary conditions defined on \emph{connected} sets and this assumption is crucial in the proof. 
It would be interesting to see if the connectedness assumption can be removed somehow.
\end{remark}

The following result can be derived from the theorem above in combination with an idea from~\cite{chen2021spectral}.
\begin{lemma}\label{lem:bounded hom}
Let $\Delta\geq 3$ and $q\geq 2$ be integers and let for some $\alpha=\alpha_\Delta<2\pi/3\Delta$,
\[
\delta_\Delta:=\sin(\alpha/2)\cos(\alpha\Delta/2).
\]
Choose $\eta>0$, $\eps>0$ and let $A\in \mathbb{R}^{q\times q}$ be a symmetric matrix such that $|A_{i,j}-1| \leq \frac{\delta_\Delta}{ (1+\eps)(1+\eta)}$ for all $i,j=1,\ldots,q$.
Then for any connected graph $G=(V,E)\in \mathcal{G}_\Delta$, any vertex $v\in V$ and any $i\in [q]$ and any boundary condition $\sigma$ on $\Lambda\subset V\setminus \{v\}$
\[
|P^\sigma_{G,v,i;A}(z)|\leq 1/\eps.
\]
for all $z\in \mathbb{D}_{1+\eta}$.
\end{lemma}
\begin{proof}
We argue by contradiction. Suppose that for some $z\in \mathbb{D}_{1+\eta}$, $P:=P^\sigma_{G,v,i;A}(z)$ satisfies $|P|>1/\eps$. 

Recall the definition of the partition function with boundary conditions $\xi\in \C^{V\times [q]}$ \eqref{eq:graph hom def external}.
Orient the edges of $G$. 
For $\xi_{u,j}\in B(1,\eps)$ for each $u\in V$ and $j\in [q]$, to be determined later, define for each edge $e=(u,w)$ a matrix $B^e$ by
\[
B^e_{i,j}=1+z(A_{i,j}-1)\cdot \xi_{u,i}^{1/\deg(u)}\xi_{w,j}^{1/\deg(w)}
\]
for $i,j=1,\ldots,q$.
Now we set $\xi_{u,j}=1$ unless $u=v$ and $j=i$ in which case we set $\xi_{v,i}=1-1/P\in B(1,\eps)$.
By construction, the matrices $B^e$ satisfy the condition of Theorem~\ref{thm:zero-free hom barvinok}
and hence
\[
Z^\sigma_{G}((B^e)_{e\in E}))\neq 0. 
\]
However, expanding the sum over all possible colors of the vertex $v$, we get
\[
Z^\sigma_{G}((B^e)_{e\in E}))=\sum_{j=1}^qZ^{\sigma_{v,j}}_G(J+z(A-J))-1/P Z^{\sigma_{v,i}}_G(J+z(A-J)=0,
\]
by definition of the ratio $P$. This is clearly a contradiction and finishes the proof.
\end{proof}

The proof of Theorem~\ref{thm:main hom} now follows quickly. 
\begin{proof}[Proof of Theorem~\ref{thm:main hom}]
Using Lemma~\ref{lem:bounded hom} combined with Lemma~\ref{lem:bounded implies SSM hom} the desired result is immediate.
\end{proof}

\begin{remark}
Using Theorem~\ref{thm:montel} (Montel's theorem) it is possible to prove a version of Lemma~\ref{lem:bounded hom} only requiring univariate zero-freeness as opposed to the possibly stronger notion of multivariate zero-freeness.
A sufficient condition would for example be that the numerator and denominator in the definition of the ratio are nonzero as well as that their difference is nonzero, so that the ratio avoids the points $0,1$ and $\infty$. 
(Various variations are possible since Theorem~\ref{thm:montel} (Montel's theorem) is quite flexible to use.)
Conceivably this could lead to better bounds for specific matrices $A$, but we are not aware of any concrete examples.
\end{remark}

\section{Concluding remarks}\label{sec:conclude}
As mentioned in the introduction our approach is quite robust and is applicable to many other models as well. The two examples that were covered in the previous sections essentially suggest a recipe for proving strong spatial mixing from absence of complex zeros. Roughly the steps are as follows.
\begin{enumerate}
\item Express the conditional probability as a rational function and bound this function using knowledge about absence of complex zeros of the partition function (with boundary conditions) either using Theorem~\ref{thm:montel} (Montel's theorem) or a variant of Lemma~\ref{lem:bounded hom}.
    \item Express the partition function of the model as the multivariate independence polynomial of an auxiliary graph with suitable weights.
    \item Use the cluster expansion to obtain a combinatorial interpretation of the coefficients of the series expansion of the rational function and show that the $k$th coefficient depends only on the depth $O(k)$ neighbourhood of the root vertex.
\end{enumerate}

It would be very interesting to know if strong spatial mixing with exponential rate implies absence of zeros in some qualitative sense. 
We expect some version of this implication to be true, but for now we refrain from making any bold conjectures.
Instead, we state a concrete question for the independence polynomial, but the question is equally interesting for other models as well.

\begin{question}
Let $\mathcal{G}$ be an infinite family of bounded degree graphs. Suppose there exist constants $r>1$ and $\lambda^\star>0$ such that the hard-core measure at any $\lambda\in (0,\lambda^\star]$ satisfies strong spatial mixing with exponential rate $r$ on $\mathcal{G}$. Does there exist an open set $U\subset \mathbb{C}$ containing $[0,\lambda^\star]$ such that for all $G\in \mathcal{G}$ and $\lambda\in U$, $Z_G(\lambda)\neq 0$?
\end{question}

Another interesting question can be found when looking at colorings of trees:
\begin{question}
Consider the $q\times q$ matrix $J-I$, where $I$ denotes the identity matrix. The partition function $Z_G(J-I)$, is equal to the number of proper $q$-colorings of the graph $G$.
It was recently shown that $\mu_{(J-I)}$ satisfies strong spatial mixing on the collection of all trees of maximum degree at most $\Delta$ provided $q\geq 1.59 \Delta$~\cite{SSMcoloringtree}.
It is however only known that there exists some $\eps>0$ such that $Z_T(J-I+zI)\neq 0$ for all $z$ in an $\eps$-neighbourhood of the unit interval and all trees of maximum degree at most $\Delta$ \emph{with} boundary conditions for $q\geq 2\Delta$~\cite{LSScorrelation}. 
Can the constant $2$ be replaced by $1.59$? So as to match the strong spatial mixing result.
\end{question}

\section*{Acknowledgement}
I thank Will Perkins for stimulating and insightful discussions. I moreover thank David Gamarnik and Tyler Helmuth and two anonymous referees for spotting some inaccuracies in an earlier version of the paper and for some useful suggestions. 
\small{
\bibliography{zeros}
\bibliographystyle{plain}
}

\end{document}